\documentclass[12pt]{article}
\usepackage[centertags]{amsmath}
\usepackage{times,fullpage}
\usepackage{latexsym,amscd,amssymb}
 \input xy
  \xyoption{all}

\usepackage[palatino,gill,courier]{altfont}

\usepackage{mathrsfs}
\usepackage{amsfonts}
\usepackage{enumitem}
\usepackage{amsmath, amsfonts, amssymb, amsthm, amscd}
\usepackage[all]{xy}

\usepackage{amsfonts}
\usepackage{amssymb}
\usepackage{amsthm}
\usepackage{newlfont}
\usepackage[colorlinks,linkcolor=blue, pdfstartview=FitH]{hyperref}
\usepackage{enumerate}
\numberwithin{equation}{section}

 \theoremstyle{plain}
\newtheorem{thm}{Theorem}[section]
\newtheorem{theorem}[thm]{Theorem}
\newtheorem{lemma}[thm]{Lemma}
\newtheorem{corollary}[thm]{Corollary}
\newtheorem{proposition}[thm]{Proposition}

\theoremstyle{definition}

\newtheorem{remark}[thm]{Remark}

\newtheorem{definition}[thm]{Definition}

\newtheorem{example}[thm]{Example}

\newtheorem{defn-thm}[thm]{Definition-Theorem}

\newcommand{\Hom}{{ Hom}}

\newcommand{\btheorem}{\begin{theorem}}
\newcommand{\etheorem}{\end{theorem}}
\newcommand{\bproposition}{\begin{proposition}}
\newcommand{\eproposition}{\end{proposition}}
\newcommand{\bdefinition}{\begin{definition}}
\newcommand{\edefinition}{\end{definition}}
\newcommand{\bcorollary}{\begin{corollary}}
\newcommand{\ecorollary}{\end{corollary}}
\newcommand{\bproof}{\begin{proof}}
\newcommand{\eproof}{\end{proof}}
\newcommand{\bremark}{\begin{remark}}
\newcommand{\eremark}{\end{remark}}
\newcommand{\eexample}{\end{example}}
\newcommand{\bexample}{\begin{example}}

\newcommand{\elemma}{\end{lemma}}
\newcommand{\blemma}{\begin{lemma}}

\renewcommand{\bar}{\overline}

\renewcommand{\phi}{\varphi}

\newcommand{\ee}{\end{eqnarray*}}
\newcommand{\be}{\begin{eqnarray*}}

\newcommand{\beq}{\begin{equation}}
\newcommand{\eeq}{\end{equation}}

\newcommand{\bd}{\begin{enumerate} [(1)]}
\newcommand{\ed}{\end{enumerate}}

\renewcommand{\tilde}{\widetilde}

\renewcommand{\bf}{\textbf}

\renewcommand{\bf}{\textbf}

\renewcommand{\Hom}{\text{Hom}}

\title{Remarks on the Chern Classes of Calabi-Yau Moduli}
\author{Kefeng Liu, Changyong Yin}
\AtEndDocument{\bigskip{\footnotesize%
\noindent \textsc{Center of Mathematical Sciences, Zhejiang University, Hangzhou, Zhejiang 310027, China;
        Department of Mathematics,University of California at Los Angeles,
        Los Angeles, CA 90095-1555, USA} \par
\textit{E-mail address:} \texttt{liu@math.ucla.edu, liu@cms.zju.edu.cn} \par
~\\~
\noindent \textsc{Department of Mathematics,University of California at Los Angeles,
        Los Angeles, CA 90095-1555, USA} \par
\textit{E-mail address:} \texttt{changyong@math.ucla.edu}
}}
\date{\today}
\begin{document}
\maketitle
\begin{abstract}
\noindent We prove that the first Chern form of the moduli space of polarized Calabi-Yau manifolds,
with the Hodge metric or the Weil-Petersson metric, represents the first Chern class of the canonical extensions
of the tangent bundle to the  compactification of the moduli space with normal crossing divisors.
\end{abstract}

\maketitle
\tableofcontents
\section{Introduction}

A compact projective manifold $X$ of complex dimension $n$ with $n
\geq 3$ is called a Calabi-Yau manifold in this paper, if it has a trivial
canonical bundle and satisfies $H^i(X, \mathcal{O}_X)=0$ for $0 < i < n$.
A polarized Calabi-Yau manifold is a pair $(X, L)$ consisting of a
Calabi-Yau manifold $X$, and an ample line bundle $L$ over $X$.
A basis of the quotient space
$(H_n(X,\mathbb{Z})/\text{Tor})/m(H_n(X,\mathbb{Z})/\text{Tor})$ is called a level $m$
structure with $m \geq 3$ on the polarized Calabi--Yau manifold.

\noindent We will consider the moduli space $\mathcal{M}_m$ of polarized Calabi-Yau manifolds with level $m$ structure with $m\geq 3$,
which is called the Calabi-Yau moduli in this paper for simplicity.
Over $\mathcal{M}_m$, we can
construct various Hodge bundles.
The holomorphic bundle $H^n$ over $\mathcal{M}_m$ whose fibers are the primitive
cohomology group $PH^n(X_p, \mathbb{C}), p \in \mathcal{M}_m$,
endowed with the Gauss-Manin connection, carry a polarized Hodge
structure of weight $n$. Then the holomorphic bundle $(H^n)^* \otimes H^n \rightarrow \mathcal{M}_m$
defines a variation of polarized Hodge structure over $\mathcal{M}_m$, which is defined over $\mathbb{Z}$.
Then, with the Hodge metric, we have the following useful observation,

\begin{theorem}\label{subbundle}
Let $\mathcal{M}_m$ be the moduli space of polarized Calabi-Yau manifolds with level $m$ structure with $m\geq 3$.
Then $(H^n)^* \otimes H^n$
defines a variation of polarized Hodge structure over $\mathcal{M}_m$, which is defined over $\mathbb{Z}$.
Moreover, with the natural Hodge metric over the Calabi-Yau moduli $\mathcal{M}_m$, the tangent bundle
\begin{eqnarray}
T \mathcal{M}_m \hookrightarrow (H^n)^* \otimes H^n,
\end{eqnarray}
is a holomorphic subbundle of $(H^n)^* \otimes H^n$ over $\mathcal{M}_m$ with the induced Hodge metric.
\end{theorem}

\noindent Then, by the important results for the integrability of Chern forms of
subbundles and quotient bundles of a variation of polarized Hodge structure over a quasi-projective manifold
as given by \cite{cks} and \cite{kollar}, see Theorem \ref{Chern forms 1} and Theorem \ref{Chern forms 2},
we can get that the first Chern form of the Calabi-Yau moduli
are integrable with the induced Hodge metric.
 More precisely, let $\widetilde{T \mathcal{M}}_m\longrightarrow \bar{\mathcal{M}}_m$ be the canonical extension of the tangent bundle of $\mathcal{M}_m$, then we have,
\begin{theorem}\label{hodgechern}
The first Chern form of the Calabi-Yau moduli $\mathcal{M}_m$ with the induced Hodge metric define currents over
 the compactification $\bar{\mathcal{M}}_m$ with normal crossing boundary divisors. Moreover,
 let $R_H$ be the curvature form of $T\mathcal{M}_m$ with the induced Hodge metric, then we have
 \begin{eqnarray*}
 \left( \frac{-1}{2 \pi i}\right)^N \int_{T\mathcal{M}_m} (tr R_H)^N= c_1(\tilde{T\mathcal{M}_m})^N
 \end{eqnarray*}
 where $N= \dim_{\mathbb{C}} \mathcal{M}_m$.
\end{theorem}

\noindent Another direct and easy consequence is that the Chern forms of the Hodge
bundles on the Calabi-Yau moduli with their induced Hodge metrics are all integrable.

\noindent In this paper we focus on Calabi-Yau manifolds.
Actually our method only needs the fact that the moduli space of the manifolds with certain
structures are smooth and quasi-projective and the period map is locally injective (the local Torelli theorem).
So our results can be easily extended to more general projective
 manifolds, including Calabi-Yau manifolds, Hyperk\"ahler manifolds, many hypersurfaces and complete intersections in projective spaces.
Here, we only summarize the results into the following theorem:

\begin{theorem}
Let  $\mathcal{M}$ be the moduli space of polarized projective manifolds
 with certain structure. Assume that $\mathcal{M}$ is smooth and quasi-projective.
If the period map from $\mathcal{M}$ into the period domain is locally injective, then the first Chern form of the moduli space $\mathcal{M}$
with the induced Hodge metric defines currents over the compactification $\bar{\mathcal{M}}$ with
normal crossing boundary divisors. Moreover, the first Chern form represents the first Chern class
of the corresponding canonical extension $\widetilde{T \mathcal{M}}\longrightarrow \bar{\mathcal{M}}$ of the tangent bundle.
\end{theorem}

\noindent By similar argument, one can show that the Chern forms of the moduli space $\mathcal{M}$ with the Weil-Petersson metric
define currents over the compactification $\bar{\mathcal{M}}$ of the moduli space $\mathcal{M}$, and the first Chern form
 also represent the first Chern class of the
corresponding canonical extension of the tangent bundle.

\noindent This paper is organized as follows. In Section \ref{cs}, firstly, we review
the definition of the variation of Hodge structure. Then,
some essential estimates for the degeneration of
the Hodge metric of a variation of polarized Hodge structure near a normal crossing divisor was reviewed,
 which was used to derive the integrability of the Chern forms of subbundles
  and quotient bundles of the variation of polarized Hodge structure over a quasi-projective manifold.
  In Section \ref{Hodge}, we review the definition of moduli
space of polarized Calabi-Yau manifolds with level $m$ structure with $m \geq 3$ and various Hodge bundles.
Then, by a key observation
that the tangent bundle of Calabi-Yau moduli is a subbundle
of the variation of polarized Hodge structure $(H^n)^* \otimes H^n \rightarrow \mathcal{M}_m$,
we prove that the first Chern form of the Calabi-Yau moduli $\mathcal{M}_m$ are integrable, with the Hodge metric.
In Section \ref{weil-petersson}, the Weil-Petersson geometry was reviewed.
And, by the isomorphism $T \mathcal{M}_m \cong (F^n)^* \otimes F^{n-1}/ F^n$, we show that
the Chern forms of the Calabi-Yau moduli $\mathcal{M}_m$ are integrable, equipped with the Weil-Petersson metric.
\\
\\
\noindent \textbf{Acknowledgement}:
The Authors would like to thank Professors J\'anos Koll\'ar, Zhiqin Lu, Chin-Lung Wang and Kang Zuo
for their very helpful comments on an earlier version of this paper.

\section{Chern Forms of the Hodge Bundles}\label{cs}

In Section \ref{cs1}, we review
the definition of variations of Hodge structure. In Section \ref{cs2},
some essential estimates for the degeneration of
 Hodge metric of a variation of polarized Hodge structure near a normal crossing divisor was reviewed,
 which was used to derive the integrability of the Chern forms of subbundles
  and quotient bundles of a variation of polarized Hodge structure over a quasi-projective manifold.

\subsection{Variation of Hodge Structure and Period Map}\label{cs1}

Let $H_\mathbb{R}$ be a real vector space with a $\mathbb{Z}$-structure defined by a lattice $H_\mathbb{Z}\subset H_\mathbb{R}$, and let $H_\mathbb{C}$ be the complexification of $H_\mathbb{R}$. A Hodge structure of weight $n$ on $H_{\mathbb{C}}$ is a decomposition
\begin{equation*}
H_\mathbb{C}=\bigoplus\limits_{k=0}^n H^{k,n-k},\ \  \text{with}\ \ H^{n-k,k}=\bar{{H}^{k, n-k}}.
\end{equation*}
The integers $h^{k,n-k}=\dim_\mathbb{C} H^{k,n-k}$ are called the Hodge numbers. To each Hodge structure of weight $n$ on $H_{\mathbb{C}}$, one assigns the Hodge filtration:
\begin{equation}\label{fil}
H_\mathbb{C}=F^0\supset\cdots\supset F^n,
\end{equation}
with $F^k=H^{n,0}\oplus\cdots\oplus H^{k,n-k}$ and $f^k=\dim_\mathbb{C}F^k=\sum_{i=k}^nh^{i,n-i}$.
This filtration satisfies that
\begin{align}\label{filtration}
H_\mathbb{C}=F^k\oplus\bar{{F}^{n-k+1}},\ \ \text{for } 0\leq k\leq n.
\end{align}
Conversely, every decreasing filtration (\ref{fil}), with the property (\ref{filtration}) and fixed dimensions $\dim_{\mathbb{C}}F^k=f^k$,  determines a Hodge structure $\{H^{k, n-k}\}_{k=0}^n$, with
\begin{eqnarray*}
H^{k,n-k}=F^k\cap\bar{F^{n-k}}.
\end{eqnarray*}

\noindent A polarization for a Hodge structure of weight $n$ consists of the data of a Hodge-Riemann bilinear form $Q$ over $\mathbb{Z}$, which is symmetric for even $n$, skew symetric for odd $n$, such that
\begin{eqnarray}\label{cl150}
Q(H^{k,n-k},H^{r,n-r})&=0\ \ \text{unless}\ \ k=n-r,
\end{eqnarray}
\begin{eqnarray}\label{cl160}
i^{2k-n}Q(v,\bar{v})&>0\ \ \text{if}\ v\in H^{k, n-k},v\neq 0.
\end{eqnarray}

\noindent In terms of the Hodge filtration $H_\mathbb{C}= F^0  \supseteq F^1 \supseteq \cdots  \supseteq F^n$, the relations (\ref{cl150}) and (\ref{cl160}) can be written as
\begin{eqnarray}  \label{cl50}
Q\left ( F^k,F^{n-k+1}\right )=0,
\end{eqnarray}
\begin{eqnarray}  \label{cl60}
Q\left ( Cv,\bar v\right )>0 \ \ \text{if}\ \ v\ne 0,
\end{eqnarray}
where $C$ is the Weil operator given by $Cv=i^{2k-n}v$ when $v\in H^{k,n-k}=F^k\cap\bar{F^{n-k}}$.

\noindent

\begin{definition}
Let $S$ be a connected complex manifold, a variation of polarized
Hodge structure of weight $n$ over $S$ consists of a polarized local
system $H_{\mathbb{Z}}$ of $\mathbb{Z}$-modules and a filtration
of the associated holomorphic vector bundle $H$:
\begin{equation}
\cdots \supseteq F^{k-1}\supseteq F^k \supseteq \cdots
\end{equation}
by holomorphic subbundles $F^k$ which satisfy:\\
\textbf{1.} $H=F^k \oplus \overline{F^{n-k+1}}$
as $C^{\infty}$ bundles, where the conjugation is taking relative to the local
system of real vectorspace $H_{\mathbb{R}}:=H_{\mathbb{Z}}\otimes \mathbb{R}$.\\
\textbf{2.} $\nabla(F^k) \subseteq \Omega_S^1 \otimes F^{k-1}$, where $\nabla$ denotes the flat
connection on $H$.
\end{definition}

\noindent We will refer to the holomorphic subbundles $F^k$ as the Hodge bundles of the variation
of polarized Hodge structure. And for each
$s \in S$, we have the Hodge decomposition:
\begin{equation}
H_s =\bigoplus_{k=0}^n H_s^{k,n-k};~~~~H_s^{k,n-k}=\overline{H_s^{n-k,k}}
\end{equation}
where $H^{k,n-k}$ is the $C^{\infty}$ subbundle of $H$ defined by:
\begin{equation*}
H^{k,n-k}=F^k \cap \overline{F^{n-k}}.
\end{equation*}

\noindent Starting from a variation of polarized Hodge structure of weight $n$ on a complex manifold $S$,
fixing a point $s\in S$ as reference point, we can construct the period domain $D$ and its dual $\check D$.

\noindent The classifying space or the period domain $D$ for polarized Hodge structures
with Hodge numbers $\{h^{k, n-k}\}_{k=0}^n $ is the space of all such Hodge filtrations
\begin{equation*}
D=\left \{ F^n\subset\cdots\subset F^0=H_{\mathbb{C}}\mid %
\dim F^k=f^k, \eqref{cl50} \text{ and } \eqref{cl60} \text{ hold}
\right \}.
\end{equation*}
The compact dual $\check D$ of $D$ is
\begin{equation*}
\check D=\left \{ F^n\subset\cdots\subset F^0=H_{\mathbb{C}}\mid %
\dim F^k=f^k \text{ and } \eqref{cl50} \text{ hold} \right \}.
\end{equation*}
\noindent The classifying space or the period domain $D\subset \check D$ is an open subset.

\noindent Finally recall that a period map is given by a locally liftable, holomorphic mapping
\begin{eqnarray*}
\Phi: S \longrightarrow \Gamma \texttt{\char92 } D,
\end{eqnarray*}
where $D$ is the classifying space or the period domain for polarized Hodge structures with given Hodge numbers $h^{p,q}$,
$\Gamma$ is the monodromy group.

\subsection{Degeneration of Hodge Structures}\label{cs2}

In this section, we will consider a variation of polarized
Hodge structure over $S$, where $S$ is a quasi-projective manifold with $\mbox{dim}_{\mathbb{C}} S =k$.
Let $\bar{S}$ be its compactification such that $\bar{S} -S$ is a divisor of
normal crossings.

\noindent Let $(\mathcal{U}, s) \subset \bar{S} $ be a special coordinate neighborhood, i.e., a coordinate neighborhood
isomophic to the polycylinder $\Delta^k$ such that
\begin{equation*}
S \cap \mathcal{U} \cong \{s=(s_1, \cdots, s_l,\cdots, s_k)\in \Delta^k \mid \prod_{i=1}^l s_i \neq 0\}
=(\Delta^{*})^l \times \Delta^{k-l} \}.
\end{equation*}
where $\Delta, ~ \Delta^*$ are the unit disk and the punctured unit disk in the complex plane, respectively.
Consider the period map
\begin{eqnarray*}
\Phi: (\Delta^*)^l \times \Delta ^{k-l} \longrightarrow \Gamma \texttt{\char92 } D,
\end{eqnarray*}
where $\Gamma$ is the monodromy group. Let $U$ be the upper half plane of $\mathbb{C}$. Than
$U^l \times \Delta^{k-l}$ is the universal covering space of $(\Delta^*)^l \times \Delta^{k-l}$,
and we can lift $\Phi$ to a mapping
\begin{eqnarray*}
\tilde{\Phi} : U^l \times \Delta^{k-l} \longrightarrow  D.
\end{eqnarray*}

\noindent Let $(z_1, \cdots, z_l, s_{l+1}, \cdots, s_k)$ be the coordinates of $U^l \times \Delta^{k-l}$
such that $s_i = e^{2 \pi i z_i}$ for $1 \leq i \leq l$. Corresponding to each of the first $l$ variables,
we choose a monodromy transformation $\gamma_i \in \Gamma$, such that
\begin{equation*}
\widetilde{\Phi}(z_1,\cdots, z_i+1,\cdots z_l, s_{l+1}, \cdots, s_k)
=\gamma_i(\widetilde{\Phi}(z_1,\cdots, z_i,\cdots z_l, s_{l+1}, \cdots, s_k))
\end{equation*}
holds identically in all variables. And the monodromy transformations $\gamma_i$'s commute with each other.
By a theorem of Borel(see \cite{schmid}, Lemma 4.5 on p. 230), after passing to a finite cover if necessary,
 the monodromy transformation $\gamma_i$ around
the punctures $s_i=0$ is unipotent, i.e.,

\begin{eqnarray*}
\begin{cases} (\gamma_i-I)^{m-1}=0 \\ [\gamma_i, \gamma_j]=0,\end{cases}
\end{eqnarray*}
for some positive integer $m$.
Therefore, we can define the monodromy logarithm $N_i= \log \gamma_i$ by the Taylor's expansion
\begin{equation*}
N_i=log ~\gamma_i := \sum_{j\geq 1}(-1)^{j+1}\frac{(\gamma_i-1)^j}{j}, ~~~\forall  1 \leq i \leq l,
\end{equation*}
then $N_i, 1 \leq i \leq l$ are nilpotent.  Let $(v_{.})$ be a flat multivalued basis of $H$
over $\mathcal{U} \cap S$. The formula
\begin{equation*}
(\widetilde{v_{.}})(s):=\exp~(\frac{-1}{2\pi\sqrt{-1}}\sum_{i=1}^l \log ~s_i N_i)(v_{.})(s)
\end{equation*}

\noindent give us a single-valued basis of $H$. Deligne's canonical extension $\widetilde{H}$ of $H$ over $\mathcal{U}$ is generated by this basis $(\widetilde{v_{.}})$(cf. \cite{schmid}). And we have
\begin{proposition}\label{extension}
If the local monodromy is unipotent, then the cononical extension is a vector bundle, otherwise it is a coherent sheaf.
\end{proposition}

\noindent The construction of $\widetilde{H}$ is independent of the choice of the local coordinates $s_i's$
and the flat multivalued basis $(v_{.})$. For any holomorphic subbundle $A$ of $H$, Deligne's canonical extension of $A$ is defined to be $\widetilde{A}:=\widetilde{H}\cap j_{*}A$ where $j: S \rightarrow \bar{S}$ is the inclusion map. Then we have the canonical extension of the Hodge filtration:

\begin{equation*}
\widetilde{H}=\widetilde{F}^0 \supset \widetilde{F}^1 \supset
\cdots \supset \widetilde{F}^n \supset 0,
\end{equation*}

\noindent which is also a filtration of locally free sheaves.

\noindent Let $N$ be a linear combination of $N_i, 1 \leq i \leq l$,  then $N$ defines a weight flat filtration $W_{\bullet}(N)$
of $H_{\mathbb{C}}$ (cf. ~\cite{deligne},~\cite{schmid}) by

\begin{equation*}
0 \subset \cdots W_{i-1}(N)\subset W_i(N) \subset W_{i+1}(N)\subset \cdots \subset H_{\mathbb{C}}.
\end{equation*}

\noindent Denote by $W_{\bullet}^j:=W_{\bullet}(\sum_{\alpha=1}^j N_{\alpha})$ for $j=1,\cdots,l$,  we can
choose a flat multigrading
\begin{equation*}
H_{\mathbb{C}}=\sum_{\beta_1,\cdots,\beta_l} H_{\beta_1,\cdots, \beta_l},
\end{equation*}

\noindent such that
\begin{equation*}
\bigcap_{j=1}^l W_{\beta_j}^j =\sum_{k_j\leq \beta_j} H_{k_1,\cdots,k_l}.
\end{equation*}

\noindent Let $h$ be the Hodge metric on the variation of polarized Hodge structure $H$. In the special neighborhood
$\mathcal{U}$, let $v$ be a nonzero local multivalued flat section of a multigrading component $H_{k_1,\cdots,k_l}$, then
$(\widetilde{v})(s):= \exp(\frac{-\sum_{i=1}^l \log s_{i} N_{i}}{2\pi \sqrt{-1}})v(s) $
is a local single-valued section of $\widetilde{H}$. And, there holds a norm estimate (Theorem $5.21$ in \cite{cks})
\begin{equation}
\parallel \widetilde{v}(s)\parallel_h \leq C_1(\frac{-\log \mid s_1\mid}{-\log \mid s_2 \mid})^{k_1/2}
(\frac{-\log \mid s_2\mid}{-\log \mid s_3 \mid})^{k_2/2} \cdots (-\log \mid s_l \mid )^{k_l/2},
\end{equation}

\noindent on the region
\begin{equation*}
\Xi(N_1,\cdots,N_l):=\{(s_1,\cdots, s_l,\cdots,s_k) \in (\Delta^{*})^l \times \Delta^{k-l}\mid
\mid s_1 \mid \leq \cdots \mid s_k \mid \leq \epsilon\}
\end{equation*}
for some small $\epsilon > 0$, where $C_1$ is a positive constant dependent on the ordering of
$\{N_1, N_2,\cdots, N_l\}$ and $\epsilon$. Since the number of ordering of $\{N_1, N_2,\cdots, N_l\}$
is finite, for any flat multivalued local section $v$ of $H$, there exist positive constants $C_2$ and $M_2$ such that
\begin{equation}
\parallel \widetilde{v}(s) \parallel_h \leq C_2(\prod_{i=1}^l -\log \mid s_{i}\mid)^{M_2},
\end{equation}
\noindent in the domain $\{(s_1,\cdots,s_l,\cdots, s_k) \mid 0< \mid s_i \mid < \epsilon ~(i=1,\cdots,l),
\mid s_j \mid < \epsilon~(j=l+1,\cdots, k)\}$.

\noindent Moreover, since the dual $H^{*}$ is also a variation of polarized Hodge structure, we then know
that, for any flat multivalued local section $v$ of $H$, there holds

\begin{equation}
C{'}(\prod_{i=1}^l -\log \mid s_{i}\mid)^{-M}\leq \parallel \widetilde{v}(s)\parallel_h
\leq C^{''}(\prod_{i=1}^l -\log \mid s_{i}\mid)^{M},
\end{equation}
where $C^{'}$ and $C^{''}$ both only depend on $\epsilon$.

\noindent By using this norm estimate, E. Cattani, A.  Kaplan and W. Schmid \cite{cks} get the following result
for the Chern forms of Hodge bundles over the quasi-projective manifold $S$, which is Corollary $5.23$ in \cite{cks}.

\begin{theorem}\label{Chern forms 1}
Let $S$ be a smooth variety, $\bar{S} \supset S$  be a smooth compactifiction
such that $\overline{S}-S=D_{\infty}$ is a normal crossing divisor. If $H$ is a variation of polarized Hodge structure over $S$
with unipotent monodromies around $D_{\infty}$,
then the Chern forms of Hodge metric on various Hodge bundles $F^p/F^q$ define currents on the compactification $\overline{S}$.
Moreover, the first Chern form represents the first Chern class of the
canonical extension $\widetilde{F^p/F^q} \longrightarrow \overline{S}$.
\end{theorem}

\noindent Base on this result, the proof of  \cite[Theorem 5.1]{kollar} gives us the following result for
any subbundle of the variation of Hodge structure $H$.
\begin{thm}\label{Chern forms 2}
Let $S$ be a smooth variety, $\bar{S} \supset S $ be a smooth compactifiction
such that $\overline{S}-S=D_{\infty}$ is a normal crossing divisor. Let $H$ be a variation of polarized Hodge structure over $S$
with unipotent monodromies around $D_{\infty}$ and $A$ be a vector subbundle of $H$.
Then the first Chern form of $A$ with respect to the induced Hodge mtric is integrable. Moreover,
let $R_H$ be the curvature form with the induced Hodge metric over $A$, then we have
\begin{eqnarray*}
\left(\frac{-1}{2 \pi i} \right)^n \int_S (tr R_H)^n = c_1(\tilde{A})^n,
\end{eqnarray*}
where $n= \dim_{\mathbb{C}} S$.
\end{thm}

\section{Chern Forms of Calabi-Yau Moduli with the Hodge Metric}\label{Hodge}

In Section \ref{Hodge1}, we review the definition of the moduli
space of a polarized Calabi-Yau manifold with level $m$ structure with $m \geq 3$
and various Hodge bundles over the moduli space.
In Section \ref{Hodge2} and \ref{Hodge3},  by a key observation
that the tangent bundle of the Calabi-Yau moduli is a subbundle
of the variation of polarized Hodge structure $(H^n)^* \otimes H^n \rightarrow \mathcal{M}_m$,
we get that the first Chern form of the Calabi-Yau moduli $\mathcal{M}_m$ are integrable with the Hodge metric.

\subsection{Calabi-Yau Moduli and Hodge Bundles}\label{Hodge1}

In this section, we briefly review the construction of the moduli space
of polarized Calabi-Yau manifolds with leve $m$ structure with $m \geq 3$ and its basic properties.
For the concept of Kuranishi family of compact complex manifolds,
we refer to \cite[Pages $8$-$10$]{SU}, \cite[Page 94]{P} or \cite[Page 19]{viehweg}
for equivalent definitions and more details.
If a complex analytic family
$\pi: \mathcal{X}\rightarrow S$ of compact complex manifolds is complete at each point of $S$ and versal at the point
$0 \in S$, then the family $\pi: \mathcal{X}\rightarrow S$ is called a Kuranishi family of the complex maniflod
$X=\pi^{-1}(0)$. The base space $S$ is called the Kuranishi space. If the family is complete at each point of a
neighbourhood of $0 \in S$ and versal at $0$, then this family is called a local Kuranishi family at
$0 \in S$. In particular, by definition, if the family is versal at each point of $S$, then it is local Kuranishi
at each point of the base $S$.

\noindent
A basis of
the quotient space
$(H_n(X,\mathbb{Z})/\text{Tor})/m(H_n(X,\mathbb{Z})/\text{Tor})$ is called a level $m$
structure on the polarized Calabi--Yau manifold $(X, L)$, where we always assume $m\geq 3$.
For deformations of
polarized Calabi-Yau manifolds with level $m$ structure, we have the following theorem, which is a reformulation
of \cite[Theorem 2.2]{S}. One can also look at \cite{P} and
\cite{viehweg} for more details about the construction of the moduli space of Calabi-Yau manifolds.

\begin{thm}\label{quasiproj}
Let $(X, L)$ be a polarized Calabi-Yau manifold with level $m$ structure with $m \geq 3$, then
there exists a quasi-projective complex manifold $\mathcal{M}_m$ with a versal family of Calabi-Yau
maniflods,

\begin{equation}\label{m-family}
\mathcal{X}_{\mathcal{M}_m} \longrightarrow \mathcal{M}_m,
\end{equation}
containing $X$ as a fiber, and polarized by an ample line bundle $\mathcal{L}_{\mathcal{M}_m}$ on the
versal family $\mathcal{X}_{\mathcal{M}_m}$.
\end{thm}

\noindent Let us define $\mathcal{T}_L(X)$ to  be the universal
cover of the base space $\mathcal{M}_m$ of the versal family above,
\begin{equation*}
\pi: \mathcal{T}_L(X) \longrightarrow \mathcal{M}_m
\end{equation*}
and the family
\begin{equation*}
\mathcal{U} \longrightarrow \mathcal{T}_L(X)
\end{equation*}
to be the pull-back of the family (\ref{m-family}) by the projection $\pi$, which can be considered as a family of polarized and marked Calabi-Yau manifolds. Recall that a marking on a Calabi-Yau manifold is  given by an integral basis of $H_n(X,\mathbb{Z})/\text{Tor}$. For simplicity, we will denote
$\mathcal{T}_L(X)$ by $\mathcal{T}$, which has the following property:

\bproposition\label{simplyconnected}
The Teichm\"uller space $\mathcal{T}$ is a simply connected smooth complex manifold, and the family
\begin{equation}
\mathcal{U} \longrightarrow \mathcal{T}
\end{equation}
containing $X$ as a fiber, is local Kuranishi at each point of the Teichm\"uller space $\mathcal{T}$.
\eproposition

\noindent Note that the Teichm$\ddot{u}$ller space $\mathcal{T}$ does not depend on the choice of level $m$.
In fact, let $m_1, m_2$ be two different positive integers, $\mathcal{U}_1 \rightarrow \mathcal{T}_1$ and $\mathcal{U}_2
\rightarrow \mathcal{T}_2$ are two versal families constructed via level $m_1$ and level $m_2$ respectively as
above, both of which contain $X$ as a fiber. By using the fact that $\mathcal{T}_1$ and $\mathcal{T}_2$ are
simply connected and the definition of versal families, we have a biholomorphic map $f: \mathcal{T}_1
\rightarrow \mathcal{T}_2$, such that the versal family $\mathcal{U}_1 \rightarrow \mathcal{T}_1$ is the pull-
back of the versal family $\mathcal{U}_2 \rightarrow \mathcal{T}_2$ by the map $f$. Thus these two families are
isomorphic to each other.

\noindent In this paper, we call $\mathcal{M}_m$ with $m\geq 3$ the Calabi-Yau moduli for simplicity.
Given any point $p \in \mathcal{M}_m$, the corresponding fiber $X_p$ in the versal family $\mathcal{U} \rightarrow \mathcal{T}$
is a polarized Calabi-Yau manifold $(X_p, L_p)$.
Hence, the flat holomorphic bundle $H^n$ over $\mathcal{M}_m$ whose fibers are the primitive
cohomology group $PH^n(X_p), p \in \mathcal{M}_m$,
endowed with the Gauss-Manin connection, carries a polarized Hodge
structure of weight $n$.

\noindent The flat bundle $H^n$ contains a flat real subbundle
$H^n_{\mathbb{R}}$, whose fiber corresponds to the subspaces $PH^n(X_p, \mathbb{R}) \subset PH^n(X_p)$;
and $H^n_{\mathbb{R}}$, in turn, contains a flat lattice bundle $H^n_{\mathbb{Z}}$, whose fibers are the
images of $PH^n(X_p, \mathbb{Z})$ in $PH^n(X_p, \mathbb{R})$.
Moreover, there exist $C^{\infty}$ subbundles $H^{p,q} \subset H^n$ with  $p+q=n$,
whose fibers over $p \in \mathcal{M}_m$ are $PH^{p,q}(X_p)$. For $0 \leq k \leq n$, $F^k = \oplus_{i \geq k} H^{i, n-i}$
are then holomorphic subbundles of $H^n$. Thus
the bundle $H^n$ defines a variation of polarized Hodge structure over $\mathcal{M}_m$, which is defined over $\mathbb{Z}$.
Thus, by the functorial construction of variation of polarized Hodge structure,
the holomorphic bundle $(H^n)^* \otimes H^n \rightarrow \mathcal{M}_m$
defines a variation of polarized Hodge structure over $\mathcal{M}_m$, which is defined over $\mathbb{Z}$.

\subsection{Period Map and the Hodge Metric on Calabi-Yau Moduli}\label{section period map} \label{Hodge2}
For any point $p\in  \mathcal{T}$, let $(X_p, L_p)$ be the corresponding fiber in
the versal family $ \mathcal{U} \rightarrow \mathcal{T}$, which is a polarized
Calabi--Yau manifold. A canonical identification of the middle dimensional
cohomology of $X_p$ to that of the background manifold $M$, that is,
$H^n(M)\simeq H^n(X_p)$ can be used to identify $H^n(X_p)$
for all fibers over $\mathcal{T}$. Thus we get a canonically trivial
bundle $H^n(M)\times \mathcal{T}$.
The period map from $\mathcal{T}$ to $D$
 is defined by assigning to each point $p\in\mathcal{T}$ the Hodge structure on $X_p$, that is
\begin{align*}
\Phi_{\mathcal{T}}:\, \mathcal{T} \rightarrow D, \quad\quad p\mapsto\Phi(p)=\{F^n(X_p)\subset\cdots\subset F^0(X_p)\}
\end{align*}
For the Calabi-Yau moduli $\mathcal{M}_m$,
we have the following period map:
\begin{equation}
\Phi: \mathcal{M}_m \longrightarrow  D/\Gamma,
\end{equation}
where $\Gamma$ denotes the global monodromy group which acts
properly and discontinuously on the period domain $D$. By going to finite covers of $\mathcal{M}_m$ and $D/\Gamma$, we may also assume $ D/\Gamma$ is smooth without loss of generality.

\noindent In \cite{griffiths 2}, Griffiths and Schmid studied the {Hodge metric} on the period domain $D$ which is the natural homogeneous metric on $D$. We denote it by $h$. In particular, this Hodge metric is a complete homogeneous metric.
By local Torelli theorem for Calabi--Yau manifolds, we know that $\Phi_{\mathcal{T}}, \Phi$ are both locally injective.
Thus it follows from \cite{griffiths 2} that the pull-backs of $h$ by $\Phi_{\mathcal{T}}$ and $\Phi$ on $\mathcal{T}$
and $\mathcal{M}_m$ respectively are both well-defined K\"ahler metrics. By abuse of notation, we still call these
 pull-back metrics the Hodge metrics.
 For explicit formula of the Hodge metric over moduli space of polarized Calabi-Yau manifolds, especially for
threefolds, the reader can refer to \cite{L99}, \cite{L011} and \cite{L012} for details.

\noindent The period map has several good properties, and one may refer to
Chapter~10 in \cite{Voisin} for details. Among them, one of the most
important is the following Griffiths transversality:
the period map $\Phi$ is a holomorphic map and its tangent map satisfies that
\begin{equation*}
\Phi_*(v)\in \bigoplus_{k=1}^{n}
\text{Hom}\left(F^k_p/F^{k+1}_p,F^{k-1}_p/F^{k}_p\right)\quad\text{for any}\quad p\in\mathcal{T}\ \
\text{and}\ \ v\in T_p^{1,0}\mathcal{T}
\end{equation*}
with $F^{n+1}=0$, or equivalently,
$\Phi_*(v)\in \bigoplus_{k=0}^{n} \text{Hom} (F^k_p,F^{k-1}_p).$
And, by the local Toreli theorem for
Calabi-Yau manifolds, the map $\Phi_{*}$ is injective. So we have
\begin{proposition}\label{subbundle}
Let $\mathcal{M}_m$ be the moduli space of polarized Calabi-Yau manifolds with level $m$ structure with $m \geq 3$.
Then $(H^n)^* \otimes H^n$
defines a variation of polarized Hodge structure over $\mathcal{M}_m$, which is defined over $\mathbb{Z}$.
Moreover, with the induced Hodge metric over the Calabi-Yau moduli $\mathcal{M}_m$, the tangent bundle
\begin{eqnarray}
T \mathcal{M}_m \hookrightarrow (H^n)^* \otimes H^n,
\end{eqnarray}
is a holomorphic subbundle of $(H^n)^* \otimes H^n$ over $\mathcal{M}_m$ with the induced Hodge metric.
\end{proposition}

\subsection{Chern Forms of Calabi-Yau Moduli with the Hodge Metric} \label{Hodge3}
As the Calabi-Yau moduli $\mathcal{M}_m$ is
quasi-projective, see Theorem \ref{quasiproj}, we know that
there is a compact projective manifold
$\bar{\mathcal{M}}_m$ such that $\bar{\mathcal{M}}_m- \mathcal{M}_m$ is a normal crossing divisor.
Also, the local mondromy of the variation of polarized Hodge structure
around the divisor is at least quasi-unipotent.
 Thus after passing to a finite ramified cover if necessary,
 the local monodromy becomes unipotent. Therefore, without loss of generality, we can assume
 the Hodge bundles have canonical extensions, which are vector bundles over the compactification $\bar{\mathcal{M}}_m$ of the
Calabi-Yau moduli $\mathcal{M}_m$, due to Proposition \ref{extension}.
Then, by Theorem \ref{Chern forms 1} and Theorem \ref{Chern forms 2}, we have
\begin{theorem}\label{hodgechern}
The first Chern form of the Calabi-Yau moduli $\mathcal{M}_m$ with the induced Hodge metric define currents over
 the compactification $\bar{\mathcal{M}}_m$ with normal crossing boundary divisors. Moreover,
 let $R_H$ be the curvature form of $T\mathcal{M}_m$ with the induced Hodge metric, then we have
 \begin{eqnarray*}
 \left( \frac{-1}{2 \pi i}\right)^N \int_{T\mathcal{M}_m} (tr R_H)^N= c_1(\tilde{T\mathcal{M}_m})^N
 \end{eqnarray*}
 where $N= \dim_{\mathbb{C}} \mathcal{M}_m$.
\begin{proof}
By Proposition \ref{subbundle}, with the Hodge metric, the tangent bundle $T \mathcal{M}_m$
of the Calabi-Yau moduli $\mathcal{M}_m$ is a holomorphic subbundle of the variation of polarized Hodge
structure $(H^n)^* \otimes H^n \rightarrow  \mathcal{M}_m$. So $T \mathcal{M}_m$
has the canonical extension, which give us a holomorphic vector bundle $\widetilde{T\mathcal{M}}_m\subset \widetilde{(H^n)^*\otimes H^n}$
over $\overline{\mathcal{M}}_m$.
Therefore, by Theorem \ref{Chern forms 2},
 the first Chern form of $T \mathcal{M}_m \rightarrow \mathcal{M}_m$
define currents over the compactification $\overline{\mathcal{M}}_m$ of $\mathcal{M}_m$,
which represent the first Chern class of the vector bundle $\tilde{T\mathcal{M}}_m \rightarrow \bar{\mathcal{M}}_m$ with
the induced Hodge metric.
\end{proof}
\end{theorem}

\section{Chern Forms of Calabi-Yau Moduli with the Weil-Petersson Metric} \label{weil-petersson}

In this section, we first review the Weil-Petersson geometry of the Calabi-Yau moduli. Then,
by the standard isomorphism $T \mathcal{M}_m \cong (F^n)^* \otimes F^{n-1}/ F^n$, we get that the
Chern forms of Calabi-Yau moduli $\mathcal{M}_m$ are integrable with the Weil-Petersson metric.

\noindent For each fiber $X_s$, we assign the Calabi-Yau metric $g(s)$ in the polarization K\"ahler class. Then on the fiber $X_s=X$,
the Kodaira-Spencer theory gives rise to an isomorphism $\rho : T_sS \longrightarrow H^1(X,T^{1,0}X)\cong
\mathbb{H}^{0,1}(T^{1,0}X)$, the space of the harmonic Beltrami differentials. So for $v,w \in T_s S$, one defines the Weil-Petersson metric
on $S$ by:
\begin{equation*}
g_{WP}(v,w):=\int_X \langle\rho(v), \rho(w)\rangle_{g(s)} dvol_{g(s)}
\end{equation*}
Let $\dim X=n$. Using the fact that the global holomorphic n-form $\Omega= \Omega(s)$ is flat with respect to $g(s)$, it can be shown
that
\begin{equation}\label{WP}
g_{WP}(v,w)=-\frac{\tilde{Q}(i_v \Omega, \overline{i_w\Omega})}{\tilde{Q}(\Omega, \bar{\Omega})}.
\end{equation}
Here, for convenience, we write $\tilde{Q}(\cdot,\cdot)=(\sqrt{-1})^n Q(\cdot, \cdot)$, where $Q$ is the intersection product.
Therefore, $\tilde{Q}$ has alternating signs in the successive primitive cohomology groups $PH^{p,q} \subset H^{p,q}$ with $ p+q=n$.
In particular, $g_{WP}$ is K\"ahler and is independent of the choice of $\Omega$. In fact, $g_{WP}$ is also independent of the choice of the polarization.The reader can refer to
\cite{LS} for details of the definition.

\noindent Formula (\ref{WP}) of the Weil-Pertersson metric implies
that the natural map $H^1(X,T_X)\longrightarrow \Hom (F^n, F^{n-1}/F^n)$
via the interior product $v \longmapsto i_v \Omega$ is an isometry from the
tangent bundle $T \mathcal{M}_m$ with the Weil-Petersson metric to
the Hodge bundle $(F^n)^{*}\otimes F^{n-1}/F^n$ with the induced Hodge metric.
So the Weil-Petersson metric is precisely the metric induced from the first piece
of the Hodge metric on the horizontal tangent bundle over the period domain.

\noindent More precisely,  for the Calabi-Yau moduli $\mathcal{M}_m$, we have the following period map from the moduli space to
the period domain of Hodge structures:
\begin{equation}
\Phi: \mathcal{M}_m \longrightarrow D/\Gamma
\end{equation}

\noindent And the differential of the period map $\Phi$ gives us the infinitesimal period map at $p \in \mathcal{M}_m$:
\begin{equation*}
\Phi_*: T_p \mathcal{M}_m \longrightarrow \mbox{Hom}(F^n, F^{n-1}/F^n)\oplus \mbox{Hom}(F^{n-1}/F^n, F^{n-2}/F^{n-1})\oplus \cdots.
\end{equation*}
is an isomorphism in the first piece. By using this isomorphism and Theorem \ref{Chern forms 2}, we have the following result, which is
\cite[Theorem 6.3]{LD}. Our proof is different and much simpler.
\begin{theorem}\label{curvature3}
The Chern forms of  the Calabi-Yau moduli $\mathcal{M}_m$ with the Weil-Petersson metric
 define currents over the compactification $\bar{\mathcal{M}}_m$ of $\mathcal{M}_m$. Moreover,
  the first Chern form represent the first Chern class of the quotient bundle
 $\tilde{(F^n)^* \otimes F^{n-1}}/ \tilde{(F^n)^* \otimes F^n} \rightarrow \overline{\mathcal{M}}_m$.
\begin{proof}
Equipped with the Weil-Petersson metric, the tangent bundle $T \mathcal{M}_m$ of the Calabi-Yau moduli $\mathcal{M}_m$
is isomorphic to
\begin{eqnarray*}
(F^n)^*\otimes F^{n-1}/ F^n
 \cong (F^n)^* \otimes F^{n-1} / (F^n)^* \otimes F^n,
\end{eqnarray*}
which is a quotient of subbundles of  the variation of polarized Hodge
structure $(H^n)^* \otimes H^n \rightarrow  \mathcal{M}_m$. Here the Hodge bundles $F^k$'s are all  equipped with their natural Hodge metrics. So, by Theorem \ref{Chern forms 1},
the  Chern forms of $T\mathcal{M}_m$ define currents over the compactification $\overline{\mathcal{M}}_m$ of $\mathcal{M}_m$.
Moreover, the first Chern form of the tangent bundle $T\mathcal{M}_m$ with the Weil-Petersson metric
 represent the first Chern class of the canonical extension
\begin{eqnarray*}
\tilde{T \mathcal{M}}_m \cong \tilde{(F^n)^* \otimes F^{n-1}}/ \tilde{(F^n)^* \otimes F^n}.
\end{eqnarray*}
\end{proof}
\end{theorem}

\noindent As a corollary, we have the following result on the Chern numbers,

\begin{corollary}
Let $f$ be an invariant polynomial on $\Hom(T\mathcal{M}_m, T\mathcal{M}_m)$ and $R_{WP}$
represent the curvature form of the Weil-Petersson metric on the Calabi-Yau moduli
 $\mathcal{M}_m$. Then we have
\begin{equation}\label{chernmoduli1}
\int_{\mathcal{M}_m} tr(f(R_{WP})) < \infty.
\end{equation}
\end{corollary}

\begin{proof}
The proof follows directly from Theorem \ref{curvature3}.
\end{proof}

\noindent As pointed out in the introduction, it follows from Theorem \ref{Chern forms 1} easily that the first Chern form of all of the Hodge bundles with Hodge metrics also represent the Chern classes of their canonical extensions. Finally note that the K\"ahler form of the Weil-Petersson metric is equal to the first Chern form of the Hodge bundle $F^n$ with its Hodge metric,  $$\omega_{WP} = c_1(F^n)_H,$$ so we easily deduce that the Weil-Petersson volume is finite and is a rational number, as proved in \cite{LS} and \cite{To} by computations.

\end{document}